\documentclass{amsart}
\usepackage{amssymb,amsthm,amsmath,epsfig,latexsym,xypic,supertabular}
\usepackage{calc}
\usepackage{latexsym}
\usepackage{amscd,amssymb,subfigure,hyperref}
\usepackage[arrow,matrix,graph,frame,poly,arc,tips]{xy}

\begin{document}

\newcommand{\mmbox}[1]{\mbox{${#1}$}}
\newcommand{\affine}[1]{\mmbox{{\mathbb A}^{#1}}}
\newcommand{\Ann}[1]{\mmbox{{\rm Ann}({#1})}}
\newcommand{\caps}[3]{\mmbox{{#1}_{#2} \cap \ldots \cap {#1}_{#3}}}
\newcommand{\N}{{\mathbb N}}
\newcommand{\Z}{{\mathbb Z}}
\newcommand{\R}{{\mathbb R}}
\newcommand{\KK}{{\mathbb K}}
\newcommand{\A}{{\mathcal A}}
\newcommand{\B}{{\mathcal B}}
\newcommand{\OO}{{\mathcal O}}
\newcommand{\C}{{\mathbb C}}
\newcommand{\PP}{{\mathbb P}}

\newcommand{\Tor}{\mathop{\rm Tor}\nolimits}
\newcommand{\Ext}{\mathop{\rm Ext}\nolimits}
\newcommand{\Hom}{\mathop{\rm Hom}\nolimits}
\newcommand{\im}{\mathop{\rm im}\nolimits}
\newcommand{\rk}{\mathop{\rm rk}\nolimits}
\newcommand{\codim}{\mathop{\rm codim}\nolimits}
\newcommand{\supp}{\mathop{\rm supp}\nolimits}
\newcommand{\coker}{\mathop{\rm coker}\nolimits}
\sloppy

\newtheorem{thm}{Theorem}[subsection]
\newtheorem*{thm*}{Theorem}
\newtheorem{defn}[thm]{Definition}
\newtheorem{prop}[thm]{Proposition}
\newtheorem*{prop*}{Proposition}
\newtheorem{conj}[thm]{Conjecture}
\newtheorem{lem}[thm]{Lemma}
\newtheorem{rmk}[thm]{Remark}
\newtheorem{cor}[thm]{Corollary}
\newtheorem{notation}[thm]{Notation}
\newtheorem{exm}[thm]{Example}

\newcommand{\msp}{\renewcommand{\arraystretch}{.5}}
\newcommand{\rsp}{\renewcommand{\arraystretch}{1}}

\newenvironment{lmatrix}{\renewcommand{\arraystretch}{.5}\small
  \begin{pmatrix}} {\end{pmatrix}\renewcommand{\arraystretch}{1}}
\newenvironment{llmatrix}{\renewcommand{\arraystretch}{.5}\scriptsize
  \begin{pmatrix}} {\end{pmatrix}\renewcommand{\arraystretch}{1}}
\newenvironment{larray}{\renewcommand{\arraystretch}{.5}\begin{array}}
  {\end{array}\renewcommand{\arraystretch}{1}}

\title[Inverse systems and the weak Lefschetz property]
{Inverse systems, Gelfand-Tsetlin patterns\\ and the weak Lefschetz property}

\author{Brian Harbourne}
\address{Harbourne: Mathematics Department \\ University of
  Nebraska \\
    Lincoln \\ NE 61801\\USA}
\email{bharbourne1@math.unl.edu}

\author{Hal Schenck}
\thanks{Schenck supported by NSF 07--07667}
\address{Schenck: Mathematics Department \\ University of
  Illinois \\
    Urbana \\ IL 61801\\USA}
\email{schenck@math.uiuc.edu}

\author{Alexandra Seceleanu}
\address{Seceleanu: Mathematics Department \\ University of
  Illinois \\
    Urbana \\ IL 61801\\USA}
\email{asecele2@math.uiuc.edu}


\begin{abstract}
\noindent In \cite{MMN}, Migliore--Mir\'o-Roig--Nagel show that the
Weak Lefschetz property can fail for an ideal $I \subseteq 
\KK[x_1,\ldots,x_4]$
generated by powers of linear forms. This is in contrast to the analogous
situation in $\KK[x_1,x_2,x_3]$, where WLP always holds \cite{SS}.
We use the inverse system dictionary to connect $I$ to an ideal of
fat points, and show that failure of WLP for powers of linear forms is 
connected
to the geometry of the associated fat point scheme. Recent results of 
Sturmfels-Xu in \cite{sturm} allow us to relate WLP to Gelfand-Tsetlin patterns.
\end{abstract}

\date{August 9, 2010}

\maketitle

\renewcommand{\thethm}{\thesection.\arabic{thm}}
\setcounter{thm}{0}

\section{Introduction}\label{sec:one}
Let $I \subseteq S=\KK[x_1,\ldots, x_r]$ be an ideal such that
$A=S/I$ is Artinian. Then $A$ has the {\em Weak Lefschetz Property}
(WLP) if there is an $\ell \in S_1$ such that for all
$m$, the map $\mu_{\ell}: A_m \stackrel {\cdot 
\ell}{\longrightarrow}A_{m+1}$
is either injective or surjective. We assume $char(\KK) = 0$; 
this simplifies our use of inverse systems. 
The case $r=1$ is trivial, and WLP always holds for $r=2$ \cite{HMNW}.
For $r=3$, WLP holds for ideals of generic forms \cite{A}, 
complete intersections \cite{HMNW}, ideals with semistable
syzygy bundle and certain splitting type, and almost complete 
intersections with unstable syzygy bundle \cite{BK},
certain monomial ideals \cite{MMN} and ideals generated 
by powers of linear forms \cite{SS}.
The following example of Migliore--Mir\'o-Roig--Nagel \cite{MMN} shows 
that the result of \cite{SS} can fail for $r\ge 4$, and motivates this paper.
\vskip -.5in
\begin{exm}\label{nonWLPlinforms}\rm
$\KK[x_1,x_2,x_3,x_4]/\langle 
x_1^3,x_2^3,x_3^3,x_4^3,(x_1\!+\!x_2\!+\!x_3\!+\!x_4)^3\rangle$
does not have WLP. The Hilbert function of $A$ is
$(1,4,10,15,15,6)$, and $A_3 \rightarrow A_4$ is not full rank.
\end{exm}
\noindent This example is explained by the following result, proved in \S 3.

\begin{prop} For generic forms $l_i \in S_1$ with
$A = S/\langle l_1^t,\ldots, l_n^t\rangle$ Artinian, the map $A_t 
\rightarrow A_{t+1}$
has full rank iff $(r,t,n) \not\in 
\{(4,3,5),(5,3,9),(6,3,14),(6,2,7)\}$.
\end{prop}

The failure of WLP in Example~\ref{nonWLPlinforms} stems from the fact 
that the space
of quartics in $\PP^2$ passing through five double points is nonempty:
WLP fails for
geometric reasons. We use inverse systems to translate questions about 
powers of linear forms to
questions about ideals of fatpoints. Then results of Alexander-Hirschowitz 
\cite{ah}, Nagata \cite{N}
and De Volder-Laface \cite{DL} can be applied to the syzygy bundle 
\cite{HMNW}, which
allows us to analyze WLP for $r=4$ when $n=5,6,7,8$ (see \S 4):

\begin{thm}
Let $I=\langle l_1^t,\ldots,l_n^t\rangle \subseteq \KK[x_1,x_2,x_3,x_4]$ with
$l_i \in S_1$ generic. If $n \in \{5,6,7,8 \}$, then WLP fails, respectively,
for $t \ge \{3,27,140,704\}$.
\end{thm}

This is surprising: for $I \subseteq \KK[x_1,x_2,x_3,x_4]$ generated
by general forms, Migliore and Miro-Roig show in \cite{MMgenforms} 
that the quotient ring always has WLP. It also contrasts to most known
cases of powers of linear forms: WLP always holds in the three variable 
case \cite{SS} and for complete intersections (i.e., $r=n$). The result on 
complete intersections is due to Stanley \cite{Stan}, who showed that if 
$I = \langle l_1^{t_1},\ldots,l_n^{t_n}\rangle$ is a complete intersection, 
then $S/I$ has the strong Lefschetz property.
In \S 5 we use this and results of D'Cruz-Iarrobino \cite{DCI}  to prove
\begin{thm}
For $I=\langle l_1^t,\ldots,l_{r+1}^t\rangle \subseteq \KK[x_1,\ldots, x_r]$ 
with $l_i \in S_1$ generic, $r=2k$, $k \geq 2$ and $t \gg 0$, 
WLP fails in degree $\frac{r}{2}(t-1)-1$.
\end{thm}
\noindent 
Migliore-Miro-Roig-Nagel \cite{MMRN2} have recently 
strengthened this result to hold for all $t$. 
They also obtain very precise results on WLP for 
almost complete intersections for $r=4,5$, when the 
powers of the linear forms are not uniform.
Using a result of Sturmfels-Xu on Gelfand-Tsetlin patterns,
we obtain partial results for $r$ odd. Based on our results and 
computational evidence, we believe
\begin{conj}\label{mainConj}
For $I=\langle l_1^t,\ldots,l_n^t\rangle \subseteq \KK[x_1,\ldots, x_r]$ with
$l_i \in S_1$ generic and $n \ge r+1\geq 5$, WLP fails for all $t \gg 0$.
\end{conj}
\renewcommand{\thethm}{\thesubsection.\arabic{thm}}
\setcounter{thm}{0}

\section{Background}\label{sec:two}
\subsection{Inverse systems} In \cite{ei}, Emsalem and Iarrobino proved
that there is a close connection between ideals generated by powers of 
linear
forms, and ideals of fatpoints.
Let $p_i = [p_{i1}:\cdots :p_{ir}] \in \mathbb{P}^{r-1}$, $I(p_i) =
\wp_i \subseteq R = {\KK}[y_1,\ldots, y_r]$, and
$\{ p_1, \ldots , p_n \} \subseteq \mathbb{P}^{r-1}$ be a set of distinct 
points.
A fat point ideal is an ideal of the form
$$F = \bigcap\limits_{i=1}^n \wp_i^{\alpha_i+1}\subset R.$$
Recall $S = {\KK}[x_1,\ldots, x_r]$ and let $L_{p_i} = 
\sum_{j=1}^rp_{i_j}x_j$.
Define an action of $R$ on $S$ by partial differentiation:
$y_j \cdot x_i = \partial x_i/\partial x_j$.
Since $F$ is a submodule of $R$, it acts on $S$. The set of elements
annihilated by the action of $F$ is denoted $F^{-1}$. Emsalem and
Iarrobino show that for $j \geq \max \{\alpha_i+1\}$,
$(F^{-1})_j = \langle L_{p_1}^{j-\alpha_1}, \ldots, L_{p_m}^{j-\alpha_m} 
\rangle_j$, and that
$\dim_{\KK}(F^{-1})_j = \dim_{\KK}(R/F)_j$.
This generalizes Terracini's lemma, where the $\alpha_i$ are all two.
For more on inverse systems, see \cite{g}.
\begin{thm}[Emsalem and Iarrobino, \cite{ei}]\label{invDim}
Let $F$ be an ideal of fatpoints:
$$F = \wp_1^{\alpha_1+1} \cap \cdots
\cap \wp_n^{\alpha_n+1}\subset R.$$
Then
$$(F^{-1})_j = \begin{cases} S_j & \mbox{for } j \leq \max\{\alpha_i\} \cr
                                &                         \cr
        L_{p_1}^{j-\alpha_1}S_{\alpha_1} + \cdots + 
L_{p_n}^{j-\alpha_n}S_{\alpha_n} &
                                \mbox{for }
j\geq\max\{ \alpha_i + 1\} \end{cases}
$$
and
$$
\dim _{\KK}(F^{-1})_j = \dim _{\KK}(R/F)_j.
$$
\end{thm}

The following corollary is just a special case version of Theorem 
\ref{invDim},
but one that we will use repeatedly.

\begin{cor}\label{invDimCor}
Let $t\ge1$ be an integer, let
$$J = \wp_1^{j-t+1} \cap \cdots
\cap \wp_n^{j-t+1}\subset R$$
be an ideal of fatpoints
and consider the ideal $I= \langle L_{p_1}^t, \ldots, L_{p_n}^t \rangle\subset S$.
Then
$$\dim_{\KK}I_j=\begin{cases}
\dim_{\KK}(R/J)_j & \mbox{for }j\ge t\cr
0 & \mbox{for } 0\le j < t\cr
\end{cases}$$
and hence
$$\dim_{\KK}(S/I)_j=\begin{cases}
\dim_{\KK}J_j & \mbox{for }j\ge t\cr
\binom{r-1+j}{r-1} & \mbox{for } 0\le j < t\cr
\end{cases}$$
\end{cor}

Note that to obtain the Hilbert function of a fixed ideal of linear forms, 
it
is necessary to consider an infinite family of ideals of fat points.

\begin{exm}\label{exceptionalEx}\rm Here we apply
Corollary \ref{invDimCor} to obtain the Hilbert function for
$A$ from Example~\ref{nonWLPlinforms}:

\begin{center}
\begin{supertabular}{|c|c|c|c|c|c|c|c|c|}
\hline $j$ & 0 & 1 & 2 & 3 & 4 & 5 & 6 & $\ldots$ \\
\hline $\dim_{\KK}A_j$ &
1 &  4 &  10 &  15 &  15 &  6 &  0 & $\ldots$ \\
\hline $HF(\cap_{i=1}^5 \wp_{i_{_{\vbox 
to4pt{\vfil}}}}^{j-2},j){}^{{}^{\vbox to4pt{\vfil}}}$ &
0 &  0 &  0 &  15 &  15 &  6 &  0 & $\ldots$ \\
\hline
\end{supertabular}
\end{center}
We consider the restriction of this example to $\PP^2$ in Example 
\ref{exP2}.
\end{exm}

\subsection{Blowups of points in projective space} There is a well-known 
correspondence between the
graded pieces of an ideal of fat points $F \subseteq \KK[x_1,\ldots,x_r]$ 
and
the global sections of a line bundle on the variety $X$ which is the
blow up of $\mathbb{P}^{r-1}$ at the points. We briefly review this.
Let $E_i$ be the class of the exceptional divisor
over the point $p_i$, and $E_0$ the pullback of a hyperplane on 
$\mathbb{P}^{r-1}$.
Given non-negative integers $m_i$, consider the fatpoints ideal 
$J=\wp_1^{m_1} \cap \cdots
\cap \wp_n^{m_n}\subset R$ and
let $$D = jE_0-\sum\limits_{i=1}^n m_iE_i.$$
Of course, $\dim_{\KK} J_j=h^0(\PP^{r-1},{\mathcal I}_Z(j))$, where 
${\mathcal I}_Z(j)$ is the ideal sheaf
of the fatpoints subscheme $Z$ defined by $F$. Moreover,
by \cite[Proposition 4.1.1]{h0}, $h^i(X,D)=h^i(\PP^{r-1},{\mathcal 
I}_Z(j))$ for all $i\ge0$.
Taking cohomology of the exact sequence
\[
0 \longrightarrow {\mathcal I}_Z(j) \longrightarrow \OO_{\PP^{r-1}}(j) 
\longrightarrow \OO_{Z}(j) \longrightarrow 0
\]
and using the fact that $\OO_{Z}(j)\cong \OO_{Z}$ and thus $h^0(Z, 
\OO_{Z}(j))=h^0(Z, \OO_{Z})=
\sum_i\binom{r-2+m_i}{r-1}$,
shows that
\begin{equation}\label{RiemannRoch}
h^0(X,D)=h^0({\mathcal 
I}_Z(j))=\binom{r-1+j}{r-1}-\sum_i\binom{r-2+m_i}{r-1}+h^1({\mathcal 
I}_Z(j)).
\end{equation}
In the context of Corollary \ref{invDimCor}, taking $m_i=j-t+1$ for all 
$i$ and defining
$D_j$ to be $D_j=jE_0-(j-t+1)(E_1+\cdots+E_n)$, we thus have:
\begin{equation}\label{Fminus}
\dim_{\KK}I_j=\begin{cases}
n\binom{r+j-t-1}{r-1}-h^1({\mathcal I}_Z(j))=
n\binom{r+j-t-1}{r-1}-h^1(D_j) & \mbox{for }j\ge t\cr
0 & \mbox{for } 0\le j < t\cr
\end{cases}
\end{equation}
Alternatively, this can be stated for the quotient $S/I=A$ as:
\begin{equation}\label{alt}
\dim_{\KK}A_j=\begin{cases}
h^0(D_j) & \mbox{for }j\ge t\cr
\binom{r-1+j}{r-1} & \mbox{for } 0\le j < t\cr
\end{cases}
\end{equation}
We will say that $I$ has {\em expected dimension} in degree $j$ if
either $I_j=0$ or $h^1(D_j)=0$. We say $D_j$ is {\em irregular} if
$h^1(D_j)>0$ and {\em regular} otherwise. We say $D_j$ is {\em special} if
$h^0(D_j)$ and $h^1(D_j)$ are both positive.

\begin{exm}\label{exP2}\rm Let $A$ be the quotient of $\KK[x_1,x_2,x_3]$ 
by the cubes of five general
linear forms. The corresponding five points in $\PP^2$ are general, and 
the first
interesting computation involves $D_4 = 4E_0-\sum\limits_{i=1}^5 2E_i$, 
for which we have
$$\dim_{\KK}A_4= h^0(D_4) = \binom{6}{2} -15 + h^1(D_4).$$
Since $H^0(D_4)$ contains the double of a conic through the five points,
$D_4$ is special, and in fact we have $h^0(D_4) =1=h^1(D_4)$.
\end{exm}

\subsection{WLP and the syzygy bundle}\label{WLPSB}
In \cite{HMNW}, Harima-Migliore-Nagel-Watanabe study WLP using the syzygy 
bundle:
\begin{defn}
If $I = \langle f_1,\ldots, f_n \rangle$ is $\langle x_1,\ldots,x_r 
\rangle-$primary, and $deg(f_i) = d_i$,
then the syzygy bundle ${\mathcal S}(I) = \widetilde{Syz(I)}$ is a rank 
$n-1$ bundle defined via
\begin{equation}\label{Defininges}
0 \longrightarrow Syz(I) \longrightarrow \bigoplus\limits_{i=1}^{n}S(-d_i) 
\stackrel{[f_1,\ldots,f_n]}{\longrightarrow} S \longrightarrow S/I 
\longrightarrow 0.
\end{equation}
or, equivalently, by
\begin{equation}\label{Definingses}
0 \longrightarrow Syz(I) \longrightarrow \bigoplus\limits_{i=1}^n S(-d_i) 
\longrightarrow I \longrightarrow 0
\end{equation}
\end{defn}

Let $\ell$ be a generic form in $S_1$ with $L = V(\ell)$, and $I$ an ideal
such that $A=S/I$ is Artinian. Sheafifying Equation \eqref{Defininges} and 
twisting gives
\[
0 \longrightarrow {\mathcal S}(I)(m) \longrightarrow 
\bigoplus\limits_{i=1}^{n}\OO_{\PP^{r-1}}(m-d_i) \longrightarrow 
\OO_{\PP^{r-1}}(m) \longrightarrow 0.
\]
Taking cohomology shows that
\begin{equation}\label{A=H1}
A = \bigoplus_{m \in \Z} H^1({\mathcal S}(I)(m)),
\end{equation}
since $A$ and
$\bigoplus_{m \in \Z} H^1({\mathcal S}(I)(m))$ both are direct sums of 
cokernels of the same maps on global sections.
Similarly,
\begin{equation}\label{SyzSections}
Syz(I)\simeq \bigoplus\limits_t H^0({\mathcal S}(I)(t)),
\end{equation}
since $Syz(I)$ and $\bigoplus\limits_t H^0({\mathcal S}(I)(t))$ both are 
direct sums of kernels of the same
maps on global sections. From Equation \eqref{Definingses} we also see 
that
\begin{equation}\label{ExpDim}
\dim_{\KK} I_j = \sum_i \binom{j-d_i+r-1}{r-1} - \dim_{\KK}Syz(I)_j.
\end{equation}
In case $f_i=L_{P_i}^t$ for a set of distinct points $P_i$, setting 
$D_j=jE_0-(j-t+1)(E_1+\cdots+E_n)$ and
comparing with Equation \eqref{Fminus}
shows that
\begin{equation}\label{ExpDimCor}
h^0({\mathcal S}(I)(j))=\dim_{\KK}Syz(I)_j=h^1(D_j)
\end{equation}
for $j\ge t$.

Since ${\mathcal S}(I)$ is a bundle, tensoring the
sequence
\[
0 \longrightarrow \OO_{\PP^{r-1}}(m) \longrightarrow \OO_{\PP^{r-1}}(m+1) 
\longrightarrow \OO_L(m+1) \longrightarrow 0
\]
with ${\mathcal S}(I)$ gives the exact sequence
\[
0 \longrightarrow {\mathcal S}(I)(m) \longrightarrow {\mathcal S}(I)(m+1) 
\longrightarrow {\mathcal S}(I)|_L(m+1) \longrightarrow 0.
\]
The long exact sequence in cohomology yields a sequence
\begin{equation}\label{les}
\xymatrixrowsep{15pt}
\xymatrixcolsep{25pt}
\xymatrix{
0 \ar[r] & H^0({\mathcal S}(I)(m)) \ar[r] &H^0({\mathcal S}(I)(m+1)) 
\ar[r]^{\phi_m} &H^0({\mathcal S}(I)|_L(m+1)) \ar[dll]\\
          & H^1({\mathcal S}(I)(m)) \ar[r]_{\mu_\ell} &H^1({\mathcal 
S}(I)(m+1)) \ar[r] &H^1({\mathcal S}(I)|_L(m+1))  \ar[dll]_{\psi_m}\\
          & H^2({\mathcal S}(I)(m)) \ar[r] &H^2({\mathcal S}(I)(m+1)) 
\ar[r] &\cdots.
}
\end{equation}
Surjectivity of $\mu_{\ell}$ in degree $m$ follows from injectivity of 
$\psi_m$, and
injectivity of $\mu_{\ell}$ from surjectivity of $\phi_m$.
In particular, $\mu_{\ell}$ is injective in degree $m$ if $h^0({\mathcal 
S}(I)|_L(m+1))=0$.

\begin{rmk}\label{sect2Rem}\rm
In the situation that $f_1,\ldots,f_n$ are $t^{th}$ powers of linear forms 
$L_{P_i}$,
we can understand ${\mathcal S}(I)|_L$ recursively. Without loss of 
generality, we may assume
$\ell=x_r$. Quotienting by the ideal $(\ell)\subset S$ gives an image 
ideal
$I'=I\otimes S'\subset S'=S/(\ell)$ that is itself generated by $t^{th}$ 
powers of linear forms (distinct since
$\ell$ is generic), these being the images under the quotient of the 
generators of $I$.
We let $A'$ denote $S'/I'$.
If $D_j=jE_0-(j-t+1)(E_1+\cdots+E_n)$ is the divisor on the blow up of
$\PP^{r-1}$ for the inverse system associated to
$I_j$, we will denote by $D'_j=jE'_0-(j-t+1)(E'_1+\cdots+E'_n)$ the 
divisor
on the blow up of $\PP^{r-2}$ for the inverse system associated to $I'_j$.
We also have
$Syz(I')=Syz(I)\otimes S'$ and thus ${\mathcal S}(I')={\mathcal 
S}(I)|_L={\mathcal S}(I)\otimes S'$.
Indeed, tensoring Equation \eqref{Definingses} by $S'$
yields the sequence
\begin{equation}\label{resSequence}
0 \longrightarrow Tor^S_1(I,S')\longrightarrow  Syz(I)\otimes S' 
\longrightarrow \bigoplus\limits_{i=1}^n S'(-t) \longrightarrow I\otimes 
S' \longrightarrow 0.
\end{equation}
But $Tor^S_1(I,S') = 0$ since it is the kernel of the injective map $I 
\stackrel{\mu_\ell}{\longrightarrow} I(1)$, so
\begin{equation}\label{Definingses2}
0 \longrightarrow Syz(I') \longrightarrow \bigoplus\limits_{i=1}^n S'(-d_i) 
\longrightarrow I' \longrightarrow 0
\end{equation}
is exact, analogous to Equation \eqref{Definingses}. Thus we also have
\begin{equation}\label{A=H1'}
A' = \bigoplus_{m \in \Z} H^1({\mathcal S}(I')(m))= \bigoplus_{m \in \Z} 
H^1({\mathcal S}(I)|_L(m)),
\end{equation}
\begin{equation}\label{SyzSections2}
Syz(I')=Syz(I)\otimes_S S'\simeq \bigoplus\limits_t H^0({\mathcal 
S}(I')(t)),
\end{equation}
\begin{equation}\label{ExpDim2}
\dim_{\KK} I'_j = \sum_i \binom{j-d_i+r-2}{r-2} - \dim_{\KK}Syz(I')_j,
\end{equation}
and, for $j\geq t$,
\begin{equation}\label{ExpDimCor2}
h^0({\mathcal S}(I')(j))=\dim_{\KK}Syz(I')_j=h^1(D'_j).
\end{equation}
Thus $\mu_\ell$ is injective in degree $m$
if $m+1\ge t$ and $h^1(D'_{m+1})=0$, since 
by Equation \eqref{Fminus} applied to $I'_{m+1}$ and $D'_{m+1}$ for 
$\PP^{r-2}$ we have $h^0({\mathcal S}(I)|_L(m+1))=h^1(D'_{m+1})$.
\end{rmk} 

\section{The Alexander-Hirschowitz theorem and generic 
forms}\label{sec:three}
A landmark result on the dimension of linear systems is:

\renewcommand{\thethm}{\thesection.\arabic{thm}}
\setcounter{thm}{0}

\begin{thm}[Alexander--Hirschowitz \cite{ah}]\label{AHthm}
Fix $m,r-1 \ge 2$, and consider the linear system of hypersurfaces of 
degree $m$ in $\PP^{r-1}$
passing through $n$ general points with multiplicity two. Then
\begin{enumerate}
\item For $m=2$, the system is special iff $2 \le n \le r-1$.
\item For $m$ greater than two, the only special systems are $(r-1,m,n) 
\in \{(2,4,5),(3,4,9),(4,4,14),(4,3,7)\}$.
In each of these four cases, the linear system is expected to be empty 
but in fact has projective dimension 0.
\end{enumerate}
\end{thm}

As a consequence of Theorem \ref{AHthm} and the developments from section 
\ref{sec:two}, we have
fairly complete information on WLP for quotients $A$ by ideals of powers 
of $n$ generic linear forms
when $n$ is not too small; specifically, we have:

\begin{prop}\label{AHcor}
Given generic linear forms $l_i$ such that $I=\langle l_1^t, \ldots, l_n^t\rangle$ and $A = 
{\KK}[x_1,\ldots,x_r]/I=S/I$ is Artinian,
consider the maps $\mu_{\ell}: A_j \rightarrow A_{j+1}$ where $\ell =x_r$ 
and $L$ is the
hyperplane defined by $\ell$.
\begin{itemize}
\item[(a)] For $0\le j\le t-2$, $\mu_{\ell}: A_j \rightarrow A_{j+1}$ is 
injective but not surjective.
\item[(b)] If $n \geq \binom{r-2 +t}{r-2}$, then $\mu_{\ell}: A_j 
\rightarrow A_{j+1}$ is surjective for $j\geq t-1$.
\item[(c)] The map $\mu_{\ell}: A_t \rightarrow A_{t+1}$
is injective if and only if $(r,t,n) \not\in 
\{(4,3,5),(5,3,9),(6,3,14),(6,2,7)\}$.
\item[(d)] If $n > \binom{r-2 +t}{r-2}$,
$A_{t-1}\stackrel{\mu_\ell}\rightarrow A_t$ is not injective, while if
$n \geq \binom{r-2 +t}{r-2}$, $A_t\stackrel{\mu_\ell}\rightarrow A_{t+1}$ 
is an isomorphism.
\end{itemize}
\end{prop}
\begin{proof} (a) If $j\le t-2$, then $I_{j+1}=I_j=0$ and hence $A_j=S_j$
and $A_{j+1}=S_{j+1}$, but $S$ is a domain
with $\dim_{\KK}S_j<\dim_{\KK}S_{j+1}$.

(b) Let $S'=S/(\ell)$, $I'=I|_L$ and $A'=S'/I'$. Note that since ${\rm 
char}(\KK)=0$, the
locus of $t^{th}$ powers of all linear forms in $S'$ satisfies no non-trivial 
linear relation
(this would be false if ${\rm char}(\KK)>0$ and $t$ were a power of
the characteristic). Thus the span of the $t^{th}$ powers of $n$ generic
linear forms has maximal dimension; i.e., its dimension is the minimum of
$n$ and the dimension $\binom{r-2 +t}{r-2}$ of the space of all forms of 
degree $t$
in $r-1$ variables. Since $n \geq \binom{r-2 +t}{r-2}$, we see that
$I'_{j+1}=S'_{j+1}$ for $j=t-1$ (and hence for $j\ge t-1$), hence 
$A'_{j+1}=0$ for $j\geq t-1$.

Now by Equation \eqref{A=H1'}
we have $H^1({\mathcal S}(I)|_L(j+1))=A'_{j+1}=0$ for $j\geq t-1$,  so by
Equation \eqref{les}
the map $A_j\stackrel{\mu_\ell}\rightarrow A_{j+1}$ is surjective.

(c) We now consider injectivity of $A_{t}\stackrel{\mu_\ell}\rightarrow 
A_{t+1}$.
From the long exact sequence of Equation~\eqref{les}, we have
\[
\xymatrixrowsep{25pt}
\xymatrixcolsep{12pt}
\xymatrix{0 \ar[r] & H^0({\mathcal S}(I)(t)) \ar[r] &H^0({\mathcal 
S}(I)(t+1)) \ar[r]
&H^0({\mathcal S}(I)|_L(t+1)) \ar[r] & A_t \ar[r]^{\cdot \ell} & A_{t+1}}
\]
Since $I$ is generated in degree $t$, $h^0({\mathcal S}(I)(t)) =0$.
Whenever $h^0({\mathcal S}(I)(t+1)) < h^0({\mathcal S}(I')(t+1))$ we thus
see that $\mu_{\ell}$ fails to be injective. This is precisely what occurs
if $(r,t,n) \in \{(4,3,5),(5,3,9),(6,3,14),(6,2,7)\}$. For example,
let $(r,t,n)=(4,3,5)$ and consider the divisor 
$D_{t+1}=(t+1)E_0-2(E_1+\cdots+E_n)$
on $\PP^{r-1}$ and $D'_{t+1}=(t+1)E'_0-2(E'_1+\cdots+E'_n)$ on $\PP^{r-2}$.
By Equation \eqref{RiemannRoch} we know
$h^0(D_{t+1})\ge \binom{r-1+t+1}{r-1}-n\binom{r}{r-1}=15>0$,
so $h^1(D_{t+1})=0$ by Theorem \ref{AHthm} and 
$h^0({\mathcal S}(I)(t+1))=0$ by Equation 
\eqref{ExpDimCor}, but $h^0({\mathcal 
S}(I)|_L(t+1))=h^1(D'_{t+1})>0$
by Equation \eqref{ExpDimCor2} and Theorem \ref{AHthm}.
The cases $(5,3,9),(6,3,14),(6,2,7)$ work the same way.

Now assume $(r,t,n) \not\in \{(4,3,5),(5,3,9),(6,3,14),(6,2,7)\}$.
The map $A_t \rightarrow A_{t+1}$ will be injective by Equation~\ref{les} 
if
\[
h^0({\mathcal S}(I)|_L(t+1))=0.
\]
But $h^0({\mathcal S}(I)|_L(t+1))=h^1(D'_{t+1})$
by Equation \eqref{ExpDimCor2}. Since
the restrictions of generic linear forms to $L$ remain generic,
by Theorem \ref{AHthm} we have $h^1(D'_{t+1})=0$.

(d) Assume $n > \binom{r-2 +t}{r-2}$. As shown in (b), $I'_t=S'_t$, hence
$h^0(I'(t))=h^0(S'(t)) = \binom{r-2 +t}{r-2}$. Now
by Equation \eqref{Definingses2}, using Equation \eqref{ExpDim2},
\[
h^0({\mathcal S}(I)|_L(t))=h^0(S'(0)^n)-\dim_\KK I'_t=n-\binom{r-2 +t}{r-2} > 
0.
\]
But we noted in (c) that $h^0({\mathcal S}(I)(t)) =0$.
Thus by Equation \eqref{les}, $A_{t-1}\stackrel{\mu_\ell}\rightarrow A_t$ 
is not injective.

If however $n \geq \binom{r-2 +t}{r-2}$, applying the statement of parts 
(b, c) shows that
$A_t\stackrel{\mu_\ell}\rightarrow A_{t+1}$ is an isomorphism.

\end{proof}
As pointed out to us by Iarrobino, this proposition is related to a result
of Hochster-Laksov \cite{HL}. In the situation of the proposition with 
$n = \binom{r-2 +t}{r-2}$, WLP holds at ``twin peaks''.

\begin{cor}\label{WLPfail2}
For generic linear forms $l_i$ and 
$I=\langle l_1^t, \ldots, l_n^t\rangle$ with $A = {\KK}[x_1,\ldots,x_r]/I$ 
Artinian,
the map $A_t \rightarrow A_{t+1}$ has full rank if and only if
$(r,t,n) \not\in \{(4,3,5),(5,3,9),(6,3,14),(6,2,7)\}$.
\end{cor}
\begin{proof}
By Proposition~\ref{AHcor}(c), it suffices to show in the four exceptional
cases that $\mu_{\ell}$ is not surjective.
\begin{enumerate}
\item For $I=\langle l_1^3,\ldots, l_5^3\rangle \subseteq 
\KK[x_1,\ldots,x_4]$,
the Hilbert series for $S/I$ is $(1, 4, 10, 15, 15, 6)$, as in Example 
\ref{exceptionalEx}.
But as in the proof of Proposition \ref{AHcor}(c), the kernel of
$A_3\to A_4$ has dimension $h^1(D'_{4})$, hence the cokernel has 
dimension $h^1(D'_{4})$,
so $\mu_{\ell}$ fails to have full rank, since $h^1(D'_4)=1$ by Theorem 
\ref{AHthm}.
\item Similarly, for $I=\langle l_1^3,\ldots, l_9^3\rangle \subseteq 
\KK[x_1,\ldots,x_5]$,
the Hilbert series for $S/I$ is $(1,5,15,26,25)$, and the kernel of
$A_3\to A_4$ has dimension $h^1(D'_4)=2$, so the cokernel has
dimension 1.
\item For $I=\langle l_1^3,\ldots, l_{14}^3\rangle \subseteq 
\KK[x_1,\ldots,x_6]$,
the Hilbert series for $S/I$ is $(1, 6, 21, 42, 42)$ but the kernel (and 
hence the cokernel)
of $A_3\to A_4$ has dimension $h^1(D'_4)=1$.
\item For $I=\langle l_1^2,\ldots, l_7^2\rangle \subseteq 
\KK[x_1,\ldots,x_6]$,
the Hilbert series for $S/I$ is $(1, 6, 14, 14, 5)$ but the kernel (and 
hence the cokernel)
of $A_2\to A_3$ has dimension $h^1(D'_3)=1$.
\end{enumerate}
Note that all but (2) are instances of failure of WLP at ``twin peaks''.
\end{proof}

\setcounter{thm}{0}

\section{Powers of linear forms in $\KK[x_1,x_2,x_3,x_4]$}\label{sec:four}
For powers of linear forms in $\KK[x_1,x_2,x_3]$, restriction to $\ell$ 
yields
powers of linear forms in two variables, and as shown in \cite{GS}, 
behaviour of these
ideals depends only on the degrees of the generators. This is in contrast 
to the case of
four variables, where restriction to $L=V(\ell)\simeq \PP^2$ yields powers 
of
linear forms in $\KK[x_1,x_2,x_3]$. In this section, we focus on 
powers of
linear forms in $S = \KK[x_1,\ldots,x_4]$ for which the Hilbert function 
of
the associated (restricted) fatpoint subscheme is known.

A famous open conjecture on the Hilbert function of fat points in $\PP^2$ 
is
expressed in terms of $(-1)$-curves (i.e., smooth rational curves $E$ with 
$E^2=-1$):

\begin{conj}[Segre-Harbourne-Gimigliano-Hirschowitz \cite{m1}]\label{SHGH}
Suppose that $\{p_1,\ldots,p_n\}\subseteq \PP^2$ is a collection of points 
in
general position, $X$ is the blowup of $\PP^2$ at the points, and $E_i$ 
the
exceptional divisor over $p_i$. If $F_j = jE_0-\sum_{i=1}^n a_iE_i$ is
special, then there exists a $(-1)$-curve $E$ with $E \cdot F_j \le -2$.
\end{conj}

\begin{exm}\rm Let $C = 2(2E_0-\sum\limits_{i=1}^5 E_i)+(E_0-E_1-E_2)$.
Then $h^0(C)=1$ and $h^1(C)=1$, so $C$ is special, but 
$E=2E_0-\sum\limits_{i=1}^5 E_i$
is rational by adjunction with $E^2=-1$ and $E\cdot C=-2$.
\end{exm}

\begin{lem}\label{SHGHinj}
Suppose $l_i \in S_1$ are generic and 
$I=\langle l_1^t, \ldots, l_n^t\rangle$, with $A = S/I$ Artinian.
If Conjecture~\ref{SHGH} holds and the divisor $D'_m$
corresponding to the inverse system
of $(I\otimes S/\ell)_m$ is effective but $F\cdot E\ge -1$ for all
$(-1)$-curves $E$,
then the map $A_{m-1} \rightarrow A_{m}$ is injective.
\end{lem}

\begin{proof}
By Equation \eqref{ExpDimCor2}, if $D'_m$ is nonspecial, then
$H^0({\mathcal S}(I)|_L(m))=0$. Since by Equation \eqref{les}
$H^0({\mathcal S}(I)|_L(m))$ maps onto the kernel of 
$A_{m-1}\longrightarrow A_m$,
the result follows.
\end{proof}

\begin{exm}\rm The failure of WLP for Example~\ref{nonWLPlinforms}
can be related to the occurrence of an SHGH curve $E$ which 
as we saw in the proof of Corollary \ref{WLPfail2} causes
$A_3\to A_4$ not to be injective, in this case 
$E=2E_0-E_1-\cdots-E_5$ (see Example~\ref{exP2} where we have $D'_4=2E$).
The hypothesis that $A$ is a quotient by
powers of generic linear forms is necessary for $A_3\to A_4$ to fail to be injective. 
For example, if instead $I = \langle x^3,y^3,z^3,w^3,(x+y)^3\rangle$, then 
$h^0({\mathcal S}(I)|_L(4)) =1 = h^0({\mathcal S}(I)(4))$, and 
$h^0({\mathcal S}(I)(3))=0$, so now $A_3\rightarrow A_4$ is injective.
On the other hand, injectivity can fail even when no SHGH curve occurs; 
for example, let
$I = \langle l_1^5,\ldots,l_{22}^5\rangle$ where the $l_i$ are
generic linear forms in 4 variables. Then $A_4\to A_5$ is  not injective 
by Equation \eqref{les},
since $D'_5=5E'_0-E'_1-\cdots-E'_{22}$ so $h^0({\mathcal S}(I)|_L(5)) 
=h^1(D'_5)=1$
and $h^0({\mathcal S}(I)(5)) =h^1(D_5)=0$ (because 22 general points 
impose
independent conditions on quintics on $\PP^3$ but not on $\PP^2$).
\end{exm}

The preceding example involving 22 generic linear forms
shows that the putative test $E \cdot F_j \le -2$ for irregularity for 
linear systems in Conjecture
\ref{SHGH} requires in general that $F_j$ be effective. When the number 
$n$ of general points
is at most 8 but not a square a stronger statement can be made; this is 
Lemma \ref{stronger}.
But first we find all $(-1)$-curves $E$ on $X$ when $n\le 8$.

\begin{lem}\label{possSHGH}
If $X\to\PP2$ is the blow up of distinct points $p_1,\ldots,p_8\in\PP^2$ and 
$E=dE_0-\sum_{i=1}^8 b_iE_i$ is the divisor of a $(-1)$-curve on $X$, 
then $d\leq 6$ and the $b_i$ are a permutation of one of 
the following:
$(-1,0,0,0,0,0,0,0)$ for $d=0$, $(0,0,0,0,0,0,1,1)$ for $d=1$, 
$(0,0,0,1,1,1,1,1)$ for
$d=2$, $(0,1,1,1,1,1,1,2)$ for $d=3$, $(1,1,1,1,1,2,2,2)$ for $d=4$, 
$(1,1,2,2,2,2,2,2)$
for $d=5$ and $(2,2,2,2,2,2,2,3)$ for $d=6$. Moreover, if the points $p_i$ are general,
each case does in fact give a
smooth rational curve $E$ with $E^2=-1$.
\end{lem}
\begin{proof}
It is easy to check that $E^2=-1$ in each of the cases listed in the 
statement of the lemma.
We also have $h^0(E)>0$ in each case since a naive dimension count shows 
the number of conditions imposed
by the points is always less than the dimension of the space of all forms 
of degree $d$.
Moreover, if the points are general, each divisor $dE_0-\sum_{i=1}^8 b_iE_i$ reduces by Cremona 
transformations to
either $E_0-E_1-E_2$ or $E_1$, and hence $E$ is always (linearly 
equivalent to)
a prime divisor (see \cite{N}). Adjunction now shows that $E$ is smooth 
and rational.

Now we show that the list is complete.
Since $E^2=-1$, $d^2=\sum_{i=1}^8 b_i^2 -1$, and $KE = \sum b_i -3d$,
adjunction implies $3d=\sum_{i=1}^8 b_i +1$.
By Cauchy-Schwartz,
\[
d^2\sum_{i=1}^8 b_i^2 -1\geq \frac{1}{8} (\sum_{i=1}^8 b_i)^2 -1 = 
\frac{1}{8}(3d-1)^2-1=\frac{9d^2-6d-7}{8}.
\]
Thus, $d^2-6d-7\leq 0$ so $d\in[1,7]$. However, $d=7$ forces equality of 
the $b_i$, and it is easy to see
there are no solutions. Hence $d\in\{1,2,3,4,5,6\}$, and a check shows 
only the $b_i$ above can occur.
For a different proof, see \cite{HFS}.
\end{proof}

\begin{lem}\label{stronger} Let $X$ be the blow up of $\PP^2$ at $1<n\le8$ 
general points, $n\ne4$.
Let $F$ be of the form $dE_0-m(E_1+\cdots+E_n)$ with $d\ge0$ and $m\ge0$.
Then $F$ is irregular if and only if there is a $(-1)$-curve $E$ such that 
$E\cdot F<-1$.
\end{lem}

\begin{proof} Conjecture \ref{SHGH} is known to be true for $n\le 8$ 
general points
(this follows from \cite[Theorem 9]{N}). Thus if $F$ is special (i.e., 
effective and irregular), then
there is a smooth rational curve $E$ with $E^2=-1$ such that $F\cdot 
E<-1$.
Conversely, if $F$ is effective, it is easy to see that
there being a smooth rational curve $E$ with $E^2=-1$ such that $E\cdot 
F<-1$
implies that $F$ is irregular. In particular, if $F$ is effective but $F\cdot E<0$, then 
$F-E$ is
effective. But $-K_X=3E_0-E-\cdots-E_n$ is effective since $n\le 8$,
hence $K_X-(F-E)$ is not effective, and by duality we have 
$h^2(F-E)=h^0(K_X-(F-E))=0$. However, $E$ is rational, so 
$E\cdot F<-1$
implies $h^1(E, F|_E)>0$, and the long exact sequence in cohomology coming 
from
$$0\to \OO_X(F-E)\to \OO_X(F)\to \OO_E(F)\to 0$$
now shows that $h^1(F)>0$.

In checking individual examples which we will need to do to 
handle the case that $F$ is not effective, it can be useful to note that
the same argument shows $h^1(F)>0$ when $E\cdot F<-1$ whether or not $F$ is 
effective if $h^2(F-E)=0$.  (We have $h^2(F-E)=0$ for example if $(F-E)\cdot E_0>-3$, by duality
since $h^2(F-E)=h^0(K_X-F+E)$ but $(K_X-F+E)\cdot E_0=-3+2<0$ hence 
$K_X-F+E$ is not effective.)

Two additional observations will be helpful. If $E$ is a 
$(-1)$-curve, note that
$E_0\cdot E\ge0$ and hence $F\cdot E<-1$ implies $(F-E_0)\cdot E<-1$.
Also, if $h^1(F)>0$, then $h^1(F-E_0)>0$. This is because $F\cdot 
E_0\ge0$
by hypothesis, and so $h^1(E_0,\OO_{E_0}(F))=0$. Taking cohomology of
$$0\to \OO_X(F-E_0)\to \OO_X(F)\to \OO_{E_0}(F)\to 0$$
and using $h^1(F)>0$ shows that $h^1(F-E_0)>0$. 

Now assume $F$ is not effective (and hence $m>0$); we consider each $n$ individually.
\begin{itemize}
\item{$n=1$. We must skip this case, 
since $F=-2E_1$ is irregular but $E=E_1$ is the only 
$(-1)$-curve
when $n=1$, and $E\cdot F>0$.}
\vskip .04in
\item{$n=2$. It is easy to see that 
$tE_0-m(E_1+E_2)$
is effective if and only if $t\ge m$. Thus $d<m$.
But $E=E_0-E_1-E_2$ is a $(-1)$-curve with $E\cdot (mE_0-m(E_1+E_2))=-m$ 
and
$F\cdot E=d-2m<-m$. Thus, if $m>1$, $h^1(mE_0-m(E_1+E_2))>0$ (since
$mE_0-m(E_1+E_2)$ is effective and has intersection with $E$ less than 
$-1$),
and since $F$ is obtained by subtracting off copies of $E_0$, our 
observations above
imply $F\cdot E<-1$ and $h^1(F)>0$. 
If $m=1$ then $F=-E_1-E_2$ which has $h^1=1$ (because two points fail to impose 
independent conditions on forms of degree zero) and $F\cdot E<-1$.}
\vskip .04in
\item{$n=3$. Since $N=2E_0-E_1-E_2-E_3$ is nef, $G=tE_0-m(E_1+E_2+E_3)$
is not effective if $2t<3m$ (i.e., if $G\cdot  N<0$).
On the other hand, the least $t$ such that $2t\geq 3m$ is $t=3m/2$ if $m$ 
is even
and $(3m+1)/2$ if $m$ is odd. Taking $G_0$ to be $G$ in the case that $m$ is even and $t=3m/2$, we have 
$G_0=(m/2)((E_0-E_1-E_2)+ (E_0-E_1-E_3)+(E_0-E_2-E_3))$, which
is effective, while taking $G_1$ to be $G$ when $m$ is odd and $t=(3m+1)/2$, we have 
$G_1=(2E_0-E_1-E_2-E_3)+
((m-1)/2)((E_0-E_1-E_2)+ (E_0-E_1-E_3)+(E_0-E_2-E_3))$, which also
is effective. Thus $G$ is not effective if and only if $2t<3m$. Let $E=E_0-E_1-E_2$.
Then $G_0\cdot E=-m/2$, so if $m$ is even $F=G_0-iE_0$ for some $i>0$,
and we have both $F\cdot E<-1$ and $h^1(F)>0$ if
$m>2$. If $m=2$, then $F$ either has $d=0,1$ or 2. In each case
one checks directly that $F\cdot E<-1$ and $h^1(F)>0$ both 
hold. Similarly, $G_1\cdot E=-(m-1)/2$, so if $m$ is odd then
$F\cdot E<-1$ and $h^1(F)>0$ if $m>3$. If $m=3$, then $0\le d \le 4$,
and in each case one can check that $F\cdot E<-1$ and $h^1(F)>0$.
If $m=1$, then $F\cdot E<-1$ and $h^1(F)>0$ for $d=0$ but if $d=1$, then
$F\cdot E\geq-1$ for every exceptional curve $E$ and $h^1(F)=0$.}
\vskip .04in
\item{$n=4$. We must also skip this case, since
$F=E_0-E_1-E_2-E_3-E_4$ is not effective, but
$F\cdot E\ge -1$ for every $(-1)$-curve $E$, yet
$h^1(F)=1$.}
\vskip .04in
\item{$n=5$. Let $G=tE_0-m(E_1+\cdots+E_5)$.
Note that $E=2E_0-(E_1+\cdots+E_5)$ is a $(-1)$-curve and
$N=2E_0-(E_1+\cdots+E_4)$ is nef and effective.
Thus $t\ge 2m$ implies $G=mE+iE_0$ for some $i\ge0$, and hence
$G$ is effective, while $t<2m$ implies $G\cdot N<0$,
so $G$ is not effective. Thus $d<2m$, and we have
$F\cdot E<E\cdot mE=-m$. If $m>1$, using the fact that
$h^1(mE)>0$ when $m>1$ (i.e., the effective case done above), we thus have
both $h^1(F)>0$ and $F\cdot E<-1$. If $m=1$, then
we have $d=0$ or 1, and in both cases we have
$h^1(F)>0$ and $F\cdot E<-1$.}
\vskip .04in
\item{$n=6$. Let $G=tE_0-m(E_1+\cdots+E_6)$,
$E=2E_0-(E_1+\cdots+E_5)$ and
$N=5E_0-2(E_1+\cdots+E_6)$.
Let $Q=12E_0-5(E_1+\cdots+E_6)$; note that
$Q=(2E_0-(E_1+\cdots+E_6)+E_1)+\cdots+
(2E_0-(E_1+\cdots+E_6)+E_6)$ is effective, being the sum of
six $(-1)$-curves.
Note that $N$ is effective and nef: effective since
6 double points impose at most 18 conditions on the
21 dimensional space of all quintics, and nef since
$5N=2Q+E_0$, and we check that $5N$ meets each of
the irreducible components in this sum non-negatively.
Since $G\cdot N=5t-12m$, we see if $5t<12m$, then $G$ is not effective.
On the other hand, if $5t\geq12m$, then $G$ is effective.
To see this, work mod 5; i.e., let $m=5a+i$ for $0\le i\le 4$.
The least $t$ such that $5t\geq12m$ is, respectively,
$12a$, $12a+3$, $12a+5$, $12a+8$ and $12a+10$,
and $G$ is, in turn,
$aQ$,
$aQ+(3E_0-E_1-\cdots-E_6)$,
$aQ+N$,
$aQ+N+(3E_0-E_1-\cdots-E_6)$ and
$aQ+2N$.
Each of these is effective (so if $5t\geq12m$, then $G$ is effective)
and, respectively,
$G\cdot E$ is $-a$, $-a+1$, $-a$, $-a+1$ and $-a$,
hence $G\cdot E<-1$ and $h^1(G)>0$ (and hence $F\cdot E<-1$ and 
$h^1(F)>0$,
since $F=G-iE_0$ for some $i>0$)
except when $a\le 1$ or
$G=2Q+(3E_0-E_1-\cdots-E_6)$ or $G=2Q+N+(3E_0-E_1-\cdots-E_6)$.
A direct check of the exceptional cases shows that
$F\cdot E<-1$ and $h^1(F)>0$ in each case except
$F= 2E_0-(E_1+\cdots+E_6)$ and $F=7E_0-3(E_1+\cdots+E_6)$,
and in both of these cases we have $F\cdot E=-1$ for every exceptional curve $E$
and $h^1(F)=0$.}
\vskip .04in
\item{$n=7$. Let $G=tE_0-m(E_1+\cdots+E_7)$,
$E=3E_0-(2E_1+E_2+\cdots+E_7)$ and
let $N=8E_0-3(E_1+\cdots+E_7)$.
Let $Q=21E_0-8(E_1+\cdots+E_7)$; note that
$Q=(3E_0-(E_1+\cdots+E_7)-E_1)+\cdots+
(3E_0-(E_1+\cdots+E_7)-E_7)$ is effective, being the sum of
seven $(-1)$-curves.
Note that $N$ is effective and nef: effective since
7 triple points impose at most 42 conditions on the
45 dimensional space of all octics, and nef since
$8N=3Q+E_0$, and we check that $8N$ meets each of
the irreducible components in this sum non-negatively.
Since $G\cdot N=8t-21m$, we see if $8t<21m$, then $G$ is not effective.
On the other hand, if $8t\geq21m$, then $G$ is effective.
To see this, work mod 8; i.e., let $m=8a+i$ for $0\le i\le 7$.
The least $t$ such that $8t\geq21m$ is, respectively,
$21a$, $21a+3$, $21a+6$, $21a+8$, $21a+11$, $21a+14$, $21a+16$ and 
$21a+19$,
and $G$ is, in turn,
$aQ$,
$aQ+(3E_0-E_1-\cdots-E_7)$,
$aQ+2(3E_0-E_1-\cdots-E_7)$,
$aQ+N$ and
$aQ+N+(3E_0-E_1-\cdots-E_7)$,
$aQ+N+2(3E_0-E_1-\cdots-E_7)$,
$aQ+2N$ and
$aQ+2N+(3E_0-E_1-\cdots-E_7)$.
Each of these is effective (so if $8t\geq21m$, then $G$ is effective).
But $G\cdot E$ is, respectively, $-a$, $-a+1$, $-a+2$, $-a$, $-a+1$, $-a+2$, $-a$ and $-a+1$.
Thus $h^1(F)>0$ and $F\cdot E<-1$ unless $a\le 1$, or 
$G$ is either 
$2Q+(3E_0-E_1-\cdots-E_7)$,
$2Q+2(3E_0-E_1-\cdots-E_7)$,
$3Q+2(3E_0-E_1-\cdots-E_7)$,
$2Q+N+(3E_0-E_1-\cdots-E_7)$,
$2Q+N+2(3E_0-E_1-\cdots-E_7)$,
$3Q+N+2(3E_0-E_1-\cdots-E_7)$ or
$2Q+2N+(3E_0-E_1-\cdots-E_7)$.
But in each of these exceptions (except $G=aQ$ with $a=0$, which does not give
rise to any cases of $F$), we have $(G-E_0-E)\cdot E_0>-3$ so $h^2(G-E_0-E)=0$
and, unless $G$ is either 
$0Q+2(3E_0-E_1-\cdots-E_7)$ or $0Q+N+2(3E_0-E_1-\cdots-E_7)$,
we have $(G-E_0)\cdot E<-1$, so $h^1(G-E_0)>0$
and hence $F\cdot E<-1$ and $h^1(F)>0$.
If $G=2(3E_0-E_1-\cdots-E_7)$, then $G-E_0$ has $h^1=0$ 
and $(G-E_0)\cdot E\geq -1$ for every exceptional curve $E$,
while $G-2E_0$ has $h^1>0$ and $G\cdot E<-1$.
If $G=N+2(3E_0-E_1-\cdots-E_7)$, then $G-E_0$ has $h^1=0$ 
and $(G-E_0)\cdot E\geq -1$ for every exceptional curve $E$,
while $G-2E_0$ has $h^1>0$ and $G\cdot E<-1$.}
\vskip .04in
\item{$n=8$. Let $G=tE_0-m(E_1+\cdots+E_8)$,
$E=6E_0-(3E_1+2E_2+\cdots+2E_8)$ and
let $N=17E_0-6(E_1+\cdots+E_8)$.
Let $Q=48E_0-17(E_1+\cdots+E_8)$; note that
$Q=(6E_0-2(E_1+\cdots+E_8)-E_1)+\cdots
(6E_0-2(E_1+\cdots+E_8)-E_8)$ is effective, being the sum of
eight $(-1)$-curves.
Note that $N$ is effective and nef: effective since
8 sextuple points impose at most 168 conditions on the
171 dimensional space of all 17-ics, and nef since
$17N=6Q+E_0$, and we check that $17N$ meets each of
the irreducible components in this sum non-negatively.
Since $G\cdot N=17t-48m$, we see if $17t<48m$, then $G$ is not effective.
On the other hand, if $17t\geq48m$, then $G$ is effective.
To see this, work mod 17; i.e., let $m=17a+i$ for $0\le i\le 16$.
The least $t$ such that $17t\geq48m$ is, respectively,
$48a+3i$ for $0\le i\le 5$, $48a+17+3(i-6)$ for $6\le i\le 11$,
and $48a+34+3(i-12)$ for $12\le i\le 16$, and $G$ is:
$G=aQ+i(3E_0-E_1-\cdots-E_8)$ for $0\le i\le 5$;
$G=aQ+N+(i-6)(3E_0-E_1-\cdots-E_8)$ for $6\le i\le 11$; and
$G=aQ+2N+(i-12)(3E_0-E_1-\cdots-E_8)$ for $12\le i\le 16$.
Each of these is effective (so if $17t\geq48m$, then $G$ is effective).

But $F$ is of the form $G-jE_0$ for some $j\geq1$ and some $G$ on this 
list.
Taking $j=1$ for each $G$ (so $F=G-E_0$), we have $(F-E)\cdot E_0>-3$
and thus $h^2(F-E)=0$ (except for the cases $G=0$ and
$G=3E_0-E_1-\cdots-E_8$, but if $G=0$ then $F\cdot E_0<0$ which we
excluded by hypothesis, and if $G=3E_0-E_1-\cdots-E_8$,
then $F=dE_0-E_1-\cdots-E_8$ for $0\le d\le2$, and in these cases
we have both $F\cdot E<-1$ and $h^1(F)>0$).
Thus whenever we have $F\cdot E<-1$ we have $h^1(F)>0$.
The only remaining cases $F=G-jE_0$ for which we do not have $F\cdot E<-1$ 
are:
$F=5(3E_0-E_1-\cdots-E_8)-E_0$ and
$F=N+5(3E_0-E_1-\cdots-E_8)-E_0$.
A direct check of these exceptional cases shows that
$F\cdot E=-1$ for every exceptional curve $E$ and $h^1(F)=0$.}
\end{itemize}
\end{proof}

\begin{lem}\label{Boundinj}
For $I=\langle l_1^t,\ldots,l_n^t\rangle \subseteq S$ with $n\leq 8$ and 
$l_i \in S_1$ generic,
the map $A_{m-1}\stackrel{\mu_\ell}\rightarrow A_m$ is injective for 
\begin{itemize}
\item $m<\lceil \frac{17(t-1)+2}{11}\rceil $ if $n=8$.
\item $m< \lceil \frac{8(t-1)+2}{5}\rceil $  if $n=7$.
\item $m< \lceil \frac{5(t-1)+2}{3}\rceil $  if $n=5,6$.
\end{itemize}
\end{lem}

\begin{proof}
If $m\le t$, then $A_{m-1}\rightarrow A_m$ is injective since $A_{m-1}=0$ 
by Corollary \ref{invDimCor}.
So suppose $m\geq t$. By Equations \eqref{les}, \eqref{SyzSections2} and 
\eqref{ExpDimCor2},
$A_{m-1}\rightarrow A_m$ is injective if $h^1(D'_m)=0$, where
$D'_m$ is the line bundle $mE'_0-(m-t+1)(\sum_{i=1}^n E'_i)$ on $\PP^2$.
By Lemma \ref{stronger}, $h^1(D'_m)=0$ if $D'_m\cdot E\ge -1$ for every
$(-1)$-curve $E=dE'_0-\sum_ib_iE'_i$.
Since $m\geq t$, we have $D'_m\cdot E'_i\ge 0$ for all $i$, so we now need 
to check
the remaining $(-1)$-curves listed in Lemma \ref{possSHGH}.
I.e., we may assume $d\geq1$. It suffices to show that
\[
md-(m-t+1)(\sum_{i=1}^8 b_i)> -2,
\]
but $\sum_{i=1}^8 b_i=3d+1$ for any $(-1)$-curve, so this simplifies to
$md-(m-t+1)(3d-1) > -2$ or $m< \frac{(3d-1)(t-1)+2}{(2d-1)}$.
The right hand side is decreasing as a function of $d$.
Thus for each $n$ we use the largest $d$ available;
i.e., $d=6$ for $n=8$, $d=3$ for $n=7$ and $d=2$ for $n=5,6$.
Plugging in these values of $d$ gives the result.
\end{proof}

\begin{lem}\label{5forms}
For $I=\langle l_1^t,\ldots,l_5^t\rangle \subseteq S$ with $l_i \in S_1$ 
generic,
$S/I$ fails to have WLP for all $t\ge 3$.
\end{lem}
\begin{proof}
For $t=3$, the result follows from Corollary~\ref{WLPfail2}.
For larger $t$, we will apply the
main result of De Volder-Laface \cite{DL} on fatpoints in $\PP^3$.
Assume $2a\geq 4b\geq0$; then the divisor $aE_0-b\sum_{i=1}^n E_i$
obtained by blowing up $n\le 8$ general points on $\PP^3$
is effective since $\binom{a+3}{3}>5\binom{b+2}{3}$ and
by \cite{DL} it is non-special since $a>2b-2$. So
for $D_m=mE_0-\sum_{i=1}^5(m-t+1)E_i$ we have $h^1(D_m)=0$
if $2m \geq 4(m-t+1)$ and $m\geq t$, or equivalently if $2t-2\geq m\geq 
t$.
So if $2t-2\geq m\geq t$ we have
\begin{equation}\label{dVL}
\dim_{\KK}A_m = h^0(D_m)=\binom{m+3}{3} -5\binom{m-t+3}{3}
\end{equation}
by Equation \eqref{alt}
and $h^0({\mathcal S}(I)(m))=h^1(D_m)=0$ by Equation \eqref{ExpDimCor}.
Now by Equation \eqref{les} we have an exact sequence
\[
0 \longrightarrow H^1(D'_m) \longrightarrow A_{m-1} \longrightarrow A_m
\]
as long as $2t-2\geq m\geq t$.

For $m = \lceil \frac{5t}{3}\rceil -1$, we have $D'_m\cdot 
(2E_0-E_1-\cdots-E_5)\leq-2$,
so $h^1(D'_m)>0$ by Lemma \ref{stronger}. (Note for this value of $m$
we have $t\le m-1<m\leq 2t-2$ for $t\geq3$.)
Thus, to prove failure of WLP, it suffices to show
\[
\dim_{\KK}A_{\lceil \frac{5t}{3}\rceil -1} \ge \dim_{\KK}A_{\lceil 
\frac{5t}{3}\rceil -2}.
\]
We obtain these dimensions from Equation \eqref{dVL}. So the result will 
follow if
\[
\binom{\lceil \frac{5t}{3}\rceil -1+3}{3} -5\binom{\lceil 
\frac{5t}{3}\rceil -1-t+3}{3} \ge
\binom{\lceil \frac{5t}{3}\rceil -2+3}{3} -5\binom{\lceil 
\frac{5t}{3}\rceil -2-t+3}{3}.
\]
A calculation shows this holds for all $t \ge 6$.
For the case $t=4$, the Hilbert function of $A$ is
$(1,4,10,20,30,36,34,20)$, and we have the sequence
\[
0 \longrightarrow H^1(D_6) \longrightarrow A_5 \longrightarrow A_6.
\]
Since $\dim_{\KK}A_5 = 36$ and $\dim_{\KK}A_6 = 34$ and by \cite{N} 
$h^1(D_6) = 3$, so
$A_5\rightarrow A_6$ has rank $33$, and WLP fails.

For $t=5$, the Hilbert function of $A$ is
$(1,4,10,20,35,51,64,70,65,45,16)$, and we have the sequence
\[
0 \longrightarrow H^1(D_8) \longrightarrow A_7 \longrightarrow A_8.
\]
Since $\dim_{\KK}A_7 = 70$ and $\dim_{\KK}A_8= 65$ and by \cite{N} 
$h^1(D_8) = 6$, so
$A_7\rightarrow A_8$ has rank $64$, and WLP fails.
\end{proof}

\begin{lem}\label{6forms}
For $I=\langle l_1^t,\ldots,l_6^t\rangle \subseteq S$ with $l_i \in S_1$ 
generic,
$S/I$ has WLP for all $t \le 14$, and fails to have WLP for all $t \ge 
27$.
\end{lem}

\begin{proof}
Let $m=\lceil \frac{5t}{3}\rceil -1$, as in the proof of Lemma 
\ref{5forms}.
Mimicking the argument there, as long as $t\geq3$ we have
\begin{equation}\label{6pteqn}
\dim_{\KK}A_{m} - \dim_{\KK}A_{m-1} = \binom{m +3}{3}
-6\binom{m -t+3}{3}  - \binom{m +2}{3} +6\binom{m -t+2}{3}.
\end{equation}
and an exact sequence
\[
0 \longrightarrow H^1(D'_m) \longrightarrow A_{m-1} \longrightarrow A_m.
\]
But by Lemma \ref{stronger},
\[
C \cdot D'_m = \begin{cases} -2 & \mbox{ if } t \mbox{ mod }3 = 0 \cr
                                                      -3 & \mbox{ if } t 
\mbox{ mod }3 = 1 \cr
                                                      -4 & \mbox{ if } t 
\mbox{ mod }3 = 2 \end{cases}
\]
where $C=2E_0'-E'_1-\cdots-E'_5$, hence $h^1(D'_m)>0$.
Since Equation \eqref{6pteqn} is positive for $t\geq48$, we see
WLP fails for $t\geq48$.
Using Lemma~\ref{Boundinj} and Proposition~2.1 of \cite{MMN} and analyzing
individual cases shows that WLP holds for all $t \le 14$, and fails for 
all $27 \le t \le 47$.
Finally, for $t=15$ WLP fails: $h^1(D_m)=6$,
$\dim_{\KK}A_{m-1}=1610$ and $\dim_{\KK}A_m=1605$,
and for $t=26$ WLP holds: $h^1(D_m)=36 = \dim_{\KK}A_{m-1}-\dim_{\KK}A_m$.
\end{proof}

\begin{thm}\label{Main2}
Let $I=\langle l_1^t,\ldots,l_n^t\rangle \subseteq \KK[x_1,x_2,x_3,x_4]$ 
with
$l_i \in S_1$ generic. If $n \in \{5,6,7,8 \}$, then WLP fails, 
respectively,
for $t \ge \{3,27,140,704\}$.
\end{thm}

\begin{proof}
Lemma~\ref{5forms} and Lemma~\ref{6forms} take care of the cases $n=5,6$.
For $n=7$ or $8$, the same argument as used in Lemmas \ref{5forms} and 
\ref{6forms} shows that
as long as $t\geq3$ we have
\begin{equation}\label{npteqn}
\dim_{\KK}A_{m} - \dim_{\KK}A_{m-1} = \binom{m +3}{3}
-6\binom{m -t+3}{3}  - \binom{m +2}{3} +n\binom{m -t+2}{3}.
\end{equation}
and an exact sequence
\[
0 \longrightarrow H^1(D'_m) \longrightarrow A_{m-1} \longrightarrow A_m.
\]
But by Lemma \ref{stronger}, for $n=7$ we have
\[
C \cdot D'_m = \begin{cases}
-5 & \mbox{ if } t \mbox{ mod }5 = 0 \cr
-2 & \mbox{ if } t \mbox{ mod }5 = 1 \cr
-4 & \mbox{ if } t \mbox{ mod }5 = 2 \cr
-6 & \mbox{ if } t \mbox{ mod }5 = 3 \cr
-3 & \mbox{ if } t \mbox{ mod }5 = 4 \cr
\end{cases}
\]
where $C=3E_0'-2E'_1-E'_2-\cdots-E'_7$, hence $h^1(D'_m)>0$, and
for $n=8$ we have
\[
C \cdot D'_m = \begin{cases}
-6 & \mbox{ if } t \mbox{ mod }11 = 0 \cr
-11 & \mbox{ if } t \mbox{ mod }11 = 1 \cr
-5 & \mbox{ if } t \mbox{ mod }11 = 2 \cr
-10 & \mbox{ if } t \mbox{ mod }11 = 3 \cr
-4 & \mbox{ if } t \mbox{ mod }11 = 4 \cr
-9 & \mbox{ if } t \mbox{ mod }11 = 5 \cr
-3 & \mbox{ if } t \mbox{ mod }11 = 6 \cr
-8 & \mbox{ if } t \mbox{ mod }11 = 7 \cr
-2 & \mbox{ if } t \mbox{ mod }11 = 8 \cr
-7 & \mbox{ if } t \mbox{ mod }11 = 9 \cr
-12 & \mbox{ if } t \mbox{ mod }11 = 10 \cr
\end{cases}
\]
where $C=6E_0'-3E'_1-2E'_2-\cdots-2E'_8$, hence again $h^1(D'_m)>0$.
Since Equation \ref{npteqn} is non-negative for $t \ge 140$ when $n=7$
and for $t \ge 704$ when $n=8$, the result follows.
\end{proof}

\noindent WLP can hold for small values of $t$,
and individual examples are easy to check:

\begin{exm}\rm
Consider $I=\left\langle l_1^8,\ldots,l_8^8\right\rangle 
\subseteq\KK[x_1,\ldots,x_4]=S$
with $l_i \in S_1$ generic, and $\ell \in S_1$ such that $I\otimes 
S/(\ell)$ is
minimally generated by powers of eight generic linear forms. Let $A=S/I$, 
$S'=S/(\ell)$,
$I|_L=I\otimes_S S'$ and let the divisor associated via the inverse system 
corresponding to $(I|_L)_m$ be
\[
D'_m=mE'_0-(m-t+1)\sum_{i=1}^8 E'_i.
\]

For degrees $\ge 8$, the Hilbert function of $A$ is:
\vskip .10in
\begin{center}
\begin{supertabular}{|c|c|c|c|c|c|c|c|c|} 
\hline $i$ & $8$ &$9$ & $10$ &$11$ &$12$& $13$&$14$&$15$\\
\hline $HF(A,i)$ & $157$ & $188$  & $206$  & $204$  & $175$  & $112$  & 
$8$  & $0$ \\
\hline
\end{supertabular}
\end{center}
\vskip .10in
By Lemma \ref{Boundinj}
the maps $A_{m-1}\stackrel{\mu_\ell}\rightarrow A_m$
are injective for $1\leq m \leq 10$. For $m=11$,
$D'_{11}\cdot (6E'_0-\sum_{i=1}^7 2E'_i - 3E'_8)<-1$
hence $h^1(D'_{11})>0$; in fact $h^1(D'_{11})=2$ \cite{h2}, giving
$A_{10}\twoheadrightarrow A_{11}$. By Proposition~2.1 of \cite{MMN},
this gives surjectivity for $m \ge 11$, so $A$ has WLP.
\end{exm}

Since Conjecture~\ref{SHGH} holds
for eight or fewer points in general position in $\PP^2$, the analysis in 
this section can be carried out for powers of eight or fewer general forms in 
$\KK[x_1,\ldots,x_4]$ where
the powers differ. In \cite{cm}, Ciliberto-Miranda show that 
Conjecture~\ref{SHGH} holds
for points with uniform multiplicity $\le 12$.
However, there is no version of the De Volder-Laface result, so
even in the special case of powers of linear forms in four variables,
the study of WLP is closely linked to a difficult open problem on 
fatpoints in $\PP^2$.

\renewcommand{\thethm}{\thesection.\arabic{thm}}
\setcounter{thm}{0}

\section{Powers of $r+1$ linear forms in $\KK[x_1,\ldots,x_r]$}\label{sec:five}
We close by tackling the case of an almost complete intersection of powers
of linear forms (so $n=r+1$). For brevity, in this section we denote
\[
\begin{array}{ccc}
A_{r,t} & = & \KK[x_1,\ldots,x_r]/\langle l_1^t,\ldots,l_{r+1}^t \rangle \\
B_{r,t} & = & \KK[x_1,\ldots,x_r]/\langle l_1^t,\ldots,l_{r+2}^t \rangle \\
C_{r,t} & = & \KK[x_1,\ldots,x_r]/\langle l_1^t,\ldots,l_{r}^t \rangle, 
\end{array}
\]
where all forms are generic. The algebras $A,B,C$ are 
related by the long exact sequence
\begin{equation}\label{4termSeq}
0 \longrightarrow (I:\ell)/I \longrightarrow S/I \stackrel{\cdot \ell}{\longrightarrow} S(1)/I \longrightarrow S(1)/I+ \langle \ell \rangle  \longrightarrow 0.
\end{equation}

\renewcommand{\thethm}{\thesubsection.\arabic{thm}}
\setcounter{thm}{0}

\subsection{A key tool} 
We now recall a key tool in analyzing WLP for $A_{r,t}$. 
Following a suggestion of Iarrobino, Stanley interprets 
$C_{r,t}$ as the cohomology ring of a product of
projective spaces and applies the Lefschetz hyperplane theorem to show that
\begin{lem}\label{Tonylemma}$[$Lemma C of \cite{I2}$]$
Let $m = \mbox{min}\{\lfloor \frac{i}{t} \rfloor,r \}$. Then 
the Hilbert function of $A_{r,t}$ in degree $i$ is 
\[
\dim_{\KK} (A_{r,t})_i =
\begin{cases}
\binom{r-1+i}{r-1} + \sum\limits_{j=1}^m(-1)^j\binom{r-1+i-tj}{r-1} \cdot \binom{r+1}{j}    & \mbox{ if this quantity is positive,} \cr
0 & \mbox{ otherwise.}\cr
\end{cases}
\]
\end{lem}

\subsection{The case of $r$ even} 
We recall that the socle degree of $B_{r-1,t}$
is the largest degree $i$ such that $\dim_{\KK} (B_{r-1,t})_i>0$.
In Lemma~2 of \cite{DCI}, D'Cruz and Iarrobino prove 

\begin{lem}\label{Tonylemma2}
For $r-1$ odd, the socle degree of $B_{r-1,t}$ is $(t-1)\frac{r}{2}$.
\end{lem}

\begin{thm}\label{FailR1} Let $k\geq 2$. Then
$A_{2k,t}$ fails to have WLP in degree $c=k(t-1)-1$ for all $t \gg 0$.
\end{thm}

\begin{proof}
By Lemma~\ref{Tonylemma2} we know $(B_{2k-1,t})_{c+1}\ne0$ and from \eqref{4termSeq}
we have the exact sequence 
\[
(A_{2k,t})_c \stackrel{\cdot \ell}{\longrightarrow} (A_{2k,t})_{c+1} \longrightarrow (B_{2k-1,t})_{c+1}  \longrightarrow 0,
\]
so WLP fails if
\[
\dim_{\KK}(A_{2k,t})_c \ge \dim_{\KK}(A_{2k,t})_{c+1}.
\]
For the relevant degrees $c$ and $c+1$, the upper limit $m$ in 
Lemma~\ref{Tonylemma} is 
\[
\begin{array}{ccc}
\mbox{For }c: m &= &\mbox{min}\{\lfloor \frac{k(t-1)-1}{t} \rfloor,2k \}\\
\mbox{For }c+1: m& =& \mbox{min}\{\lfloor \frac{k(t-1)}{t} \rfloor,2k \}
\end{array}
\]
If $t \ge k+1$, both $m$ values 
equal $k-1$, so 
by Lemma~\ref{Tonylemma} it suffices to show
\[
\begin{array}{ccc}
\binom{2k-1+c}{r-1} +\!\!\! \sum\limits_{1 \le j \le m}\!\!\! (-1)^j\binom{2k-1+c-tj}{2k-1} \cdot \binom{2k+1}{j} &\!\!\ge\!\! &\binom{2k+c}{2k-1} + \!\!\! \sum\limits_{1 \le j \le m}\!\!\! (-1)^j\binom{2k+c-tj}{2k-1} \cdot \binom{2k+1}{j} 
\end{array}
\]
Rearranging shows this inequality is equivalent to 
\[
\sum\limits_{j=0}^{k-1}(-1)^{j+1} \binom{2k-2 -k+(k-j)t}{2k-2} \cdot \binom{2k+1}{j} \ge 0.
\]
Expanding yields a polynomial of degree $2k-2$ in $t$, with lead
coefficient $\frac{\alpha}{(2k-2)!}$, where
\[
\alpha =\sum\limits_{j=0}^{k-1}(-1)^{j+1}(k-j)^{2k-2} \cdot \binom{2k+1}{j}.
\]
But $\alpha$ is the difference of two central Eulerian numbers
\begin{equation}\label{Eulerian}
\alpha =\genfrac{\langle}{\rangle}{0pt}{}{2k-2}{k-2}-\genfrac{\langle}{\rangle}{0pt}{}{2k-2}{k-3},
\end{equation}
so the positivity of $\alpha$ now follows, since the Eulerian numbers 
$\genfrac{\langle}{\rangle}{0pt}{}{n}{j}$
are increasing for $1\leq j\leq n/2$.
\end{proof}

\begin{exm}\rm
Theorem~\ref{FailR1} does not detect all obstructions to WLP. 
The Hilbert function of $A_{4,6}$ is 
\vskip .10in
\begin{center}
\begin{supertabular}{|c|c|c|c|c|c|c|c|c|c|c|c|c|c|} 
\hline $i$ & $0$ & $1$ &$2$ & $3$ &$4$ &$5$& $6$ & $7$ & $8$ & $9$ & $10$ & $11$ & $12$ \\
\hline $HF(A,i)$ & $1$ & $4$ & $10$  & $20$  & $35$  & $56$  & $79$  & $100$  & $115$ & $120$ & $111$ & $84$ & $45$ \\
\hline
\end{supertabular}
\end{center}
\vskip .10in
WLP fails for both $A_8\rightarrow A_9$ and $A_9\rightarrow A_{10}$
but only the latter failure is predicted by the theorem.
\end{exm}

\subsection{Gelfand-Tsetlin patterns} 

There is an interesting connection to combinatorics which we will apply in the next section.

\begin{defn}
A two-row Gelfand-Tsetlin pattern is a non-negative integer
$2 \times n$-matrix $(\lambda_{ij})$ that satisfies  
$\lambda_{2n} =0$, $\lambda_{1,j+1} \geq \lambda_{2,j}$ and $\lambda_{i,j} \geq
\lambda_{i,j+1}$ for $i =1,2 $ and $j=1,\ldots,n-1$. 
\end{defn}
In Proposition~3.6 of \cite{sturm}, Sturmfels-Xu show that 
for generic forms $l_i$, the Hilbert function of 
$\KK[x_1,\ldots,x_r]/\langle l_1^{u_1},\ldots,l_{r+1}^{u_{r+1}} \rangle$ in degree 
$i$ is the number of two-rowed Gelfand-Tsetlin patterns with $\lambda_{21} = i$ and $\lambda_{1j} +
\lambda_{2j} = u_j + \cdots + u_{r+1}$ for $j=1,\ldots,r+1$. 
\begin{cor}
Let $m = \mbox{min}\{\lfloor \frac{i}{t} \rfloor,r \}$. 
The number of two-rowed Gelfand-Tsetlin patterns with 
$\lambda_{21} = i$ and $\lambda_{1j} + \lambda_{2j} = (r+2-j)t$ for $j=1,\ldots,r+1$
is 
\[
\begin{cases}
\binom{r-1+i}{r-1} + \sum\limits_{1 \le j \le m}(-1)^j\binom{r-1+i-tj}{r-1} \cdot \binom{r+1}{j}    & \mbox{ if this quantity is positive.} \cr
0 & \mbox{ otherwise.}\cr
\end{cases}
\] 
\end{cor}
\begin{proof}
This follows from the result of Sturmfels-Xu and Lemma~\ref{Tonylemma}.
\end{proof}

\subsection{The case of $r=2k+1$ odd}
Let $SD(A)$ denote the socle degree of
an Artinian algebra $A$. No formula for $SD(B_{2k,t})$ 
analogous to that of Lemma~\ref{Tonylemma2} is known.
However, we can
still obtain some partial results on WLP for $A_{2k+1,t}$
by applying results from \cite{sturm}.

\begin{rmk}\label{socleB4t}\rm
The socle degree of $B_{4,t}$ up to $t=14$ is:
\vskip .10in
\begin{center}
\begin{supertabular}{|c|c|c|c|c|c|c|c|c|c|c|c|c|c|} 
\hline $t$ & $2$ & $3$ &$4$ &$5$& $6$ & $7$ & $8$ & $9$ & $10$ & $11$ & $12$ & $
13$ & $14$\\
\hline $SD(B_{4,t})$ & $2$ & $4$ & $7$  & $9$  & $12$  & $14$  & $16$  & $19$  &
 $21$ & $24$ & $26$ & $28$ & $31$ \\
\hline
\end{supertabular}
\end{center}
\vskip .10in
\end{rmk}
\noindent Lemma~3 of \cite{DCI} asserts that
$SD(B_{2k,t})=(t-1)k$ but the proof shows only that
\begin{equation}\label{DCIeqn}
(t-1)k \le SD(B_{2k,t}) \le (t-1)(k+1);
\end{equation}
the table above shows the assertion of the lemma is incorrect for $4\leq t\leq 14$. 

\begin{lem}\label{GTodd}
If $c = (t-1)(k+1)-1$ and $t>2k+2$, then 
\[\dim_{\KK}(A_{2k+1,t})_{c} \ge \dim_{\KK}(A_{2k+1,t})_{c+1}.
\]
\end{lem}
\begin{proof}
Let $G_i$ denote the set of Gelfand-Tsetlin patterns with 
$\lambda_{21} = i$ and $\lambda_{1j} +\lambda_{2j} = (2k+1+2-j)t$ for $j=1,\ldots,2k+1+1$. 
We will exhibit an injective map $G_{c+1}\rightarrow G_{c}$.
To do this, note that there is no pattern in $G_{c+1}$ with $\lambda_{22}=c+1$,
as this would imply $\lambda_{12}=(2k+1)t-c-1$. Since $\lambda_{12}\geq \lambda_{22}$ 
this yields $(2k+1)t-c-1 \geq c+1$, so $(2k+1)t-2 \geq 2c= 2[(t-1)(k+1)-1]$ and so $2k+2 \geq t$, 
a contradiction. 

Define a map $G_{c+1} \rightarrow G_c$ by sending 
$\Lambda \in G_{c+1}$ to the pattern obtained by replacing the first column
of $\Lambda$ (given by $\lambda_{11}=(2k+3)t-c-1, \lambda_{12}=c+1$) 
with $\lambda'_{11}=(2k+3)t-c, \lambda'_{12}=c$. This new filling is 
still a Gelfand-Tsetlin pattern since we have shown that 
$\lambda_{22}\leq c$, therefore the map is an injection of $G_{c+1}$ into $G_c$.
\end{proof}

We now have:

\begin{prop}\label{sect5prop}
For $A_{2k+1, 2l+1}$, with $l$ possibly a half integer, then WLP fails for
the map $(A_{2k+1,2l+1})_c \to (A_{2k+1,2l+1})_{c+1}$ if
\begin{itemize}
\item[(a)] for $c+1=2kl$ we have
$$\dim_{\KK}(A_{2k+1,2l+1})_c + \dim_{\KK}(B_{2k,2l+1})_{c+1} > \dim_{\KK}(A_{2k+1,2l+1})_{c+1},$$
or if
\item[(b)] $2l+1 > 2k+2$ and $SD(B_{2k,2l+1}) = c+1$ for $c+1=2l(k+1)$.
\end{itemize}
\end{prop}
\begin{proof} By \eqref{DCIeqn}, the socle degree of $B_{2k,2l+1}$ is at least $c+1=2kl=k(t-1)$ where $t=2l+1$,
so by \eqref{4termSeq} the map is not surjective, while if 
the stated inequality holds, then the map 
cannot by dimension considerations be injective, which proves (a).
Similarly, if the socle degree of
$B_{2k,2l+1}$ is $2l(k+1)=(k+1)(t-1)$ for $t=2l+1$, then surjectivity
fails so Lemma~\ref{GTodd} implies injectivity fails too, which proves (b).
\end{proof}

In order to apply Proposition \ref{sect5prop}(a), we will need to be able to compute
the dimension of $B_{2k,2l+1}$ in degree $c+1=2kl$. In Theorem~7.2 of \cite{sturm},
Sturmfels-Xu use the Verlinde formula to show that for generic linear 
forms $l_i$, the Hilbert function of $B_{s,2l+1}$ in degree $i=sl$ is
\begin{equation}\label{Verlinde}
\dim_{\KK}(B_{s,2l+1})_i=\frac{1}{2l+1} \sum_{j=0}^{2l} (-1)^{sj}
\bigl({\rm sin} \frac{2j+1}{4l+2} \pi \bigr)^{-s}.
\end{equation}
Here $l$ can be a half-integer if $s$ is even
but must be an integer if $s$ is odd.
In particular, the Verlinde formula gives the
Hilbert function of $B_{2k,2l+1}$ in degree $sl = 2k \cdot \frac{t-1}{2} = k(t-1)$.

When $l=1/2$ and $i=\lceil s/2\rceil$, the dimension of $(B_{s,2l+1})_i$ takes a particularly simple form:
\begin{equation}\label{decruz}
\dim_{\KK}(B_{s,2})_i=
\begin{cases}
2^i & \mbox{ if } s \mbox{ is even and } i=\frac{s}{2},\cr
1& \mbox{ if } s \mbox{ is odd and }i=\frac{s+1}{2}.\cr
\end{cases}
\end{equation}
This was conjectured by D'Cruz and Iarrobino in \cite{DCI}, 
and proved by
Sturmfels and Xu in \cite[Corollaries 7.3, 7.4]{sturm}.

\subsection{Almost complete intersections with $t=2$}
We close by studying almost complete intersections of squares of linear forms.
For example, by applying the results above we have:

\begin{exm}\label{FailQuadrics}\rm
For $B_{7,2}$, $SD(B_{7,2}) = 4$ by Lemma \ref{Tonylemma2} (and the socle dimension is $1$,
but we don't need the specific dimension in this case), while
$\dim_{\KK}(A_{8,2})_3 = 48$ and $\dim_{\KK}(A_{8,2})_4 = 42$
by Lemma \ref{Tonylemma} so WLP fails by Theorem \ref{FailR1}.
For $B_{8,2}$, $\dim_{\KK}(B_{8,2})_4=16$ from \eqref{Verlinde} 
(and $SD(B_{8,2}) = 4$ but we don't need the specific socle degree in this case), while
$\dim_{\KK}(A_{9,2})_3 = 75$ and $\dim_{\KK}(A_{9,2})_4 = 90$
by Lemma \ref{Tonylemma} so WLP fails by Proposition \ref{sect5prop}.
\end{exm}

More generally, consider the map $(A_{r,2})_{k-1}\to (A_{r,2})_{k}$ where $r$ is either $2k$ or $2k+1$.
By \eqref{DCIeqn}, the socle degree of $B_{r-1,2}$ is at least $k$, so from 
Theorem \ref{FailR1} (if $r=2k$ is even) or from Proposition \ref{sect5prop}(a)
(if $r=2k+1$ is odd), we see WLP fails if
$\dim_\KK(A_{r,2})_{k-1}\geq\dim_\KK(A_{r,2})_{k}$.
Using Lemma \ref{Tonylemma} we can check this for any specific value of $k$;
numerical experiments suggest this holds for $r\geq 15$ if r is odd and for $r\geq 6$ if $r$ is even.
If in fact $r=2k+1$ is odd, then by Proposition \ref{sect5prop}(a)
and \eqref{decruz} it is enough to show
$\dim_\KK(A_{r,2})_{k-1}+2^k>\dim_KK(A_{r,2})_{k}$.
Numerical experiments suggest this holds for odd $r\geq 9$.
This leads us to make the following conjecture.

\begin{conj}
For $A_{r,2}$, WLP fails for $r=6$ and all $r\geq 8$.
\end{conj}
In \cite{MMRN2}, Migliore-Miro-Roig-Nagel prove the
conjecture is true for an even number of variables.
\noindent{\bf Acknowledgements} Computations were performed using 
Macaulay2,
by Grayson and Stillman, available at: {\tt http://www.math.uiuc.edu/Macaulay2/}.
Scripts to analyze WLP are available at: {\tt http://www.math.uiuc.edu/$\sim$asecele2}.
Special thanks go to Pietro Majer for explaining 
Equation~\eqref{Eulerian} to us. 
\bibliographystyle{amsalpha}

\begin{thebibliography}{10}
\bibitem{ah} J. Alexander and A. Hirschowitz,
             Un lemme d'Horace diff\'erentiel: application aux singularit\'es 
             hyperquartiques de $P^5$. 
             {\em J. Algebraic Geom.}, {\bf 1} (1992), 411--426.

\bibitem{A} D. Anick,
             Thin algebras of embedding dimension three,
             {\em J. Algebra}, {\bf 100} (1986), 235--259.

\bibitem{BK} H.~Brenner, A.~Kaid,
             Syzygy bundles on $\mathbb{P}^2$ and the weak Lefschetz 
property.
             {\em Illinois J. Math.} {\bf 51}  (2007), 1299--1308.

\bibitem{cm} C. Ciliberto and R. Miranda,
              Linear systems of plane curves with base points of equal 
multiplicity,
             {\em Trans. Amer. Math. Soc.}, {\bf 352} (2000), 4037--4050.

\bibitem{DCI} C.~D'Cruz, A.~Iarrobino,
             High-order vanishing ideals at $n+3$ points of $\PP^n$.
             {\em J. Pure Appl. Algebra}, {\bf 152} (2000), 75--82.

\bibitem{DL} C. DeVolder, A. Laface,
             On linear systems of $\mathbb{P}^3$ through multiple points.
             {\em J. Algebra}, {\bf 310} (2007), 207--217.

\bibitem{ei}
         J. Emsalem and A. Iarrobino,
         Inverse system of a symbolic power I,
         {\em J. Algebra}, {\bf 174} (1995), 1080--1090.

\bibitem{g}
         A. Geramita,
         Inverse systems of fat points: Waring's problem, secant varieties
         of Veronese varieties, and parameter spaces for Gorenstein ideals,
         {\em Queens Papers in Pure and Applied Mathematics} {\bf 102} 
(1996),
         1--114.

\bibitem{GS} A.~Geramita, H.~Schenck,
             Fat points, inverse systems, and piecewise polynomial 
functions,
             {\em J. Algebra}, {\bf 204} (1998), 116--128.

\bibitem{ghi} A. Gimigliano, B. Harbourne, M. Ida,
             Betti numbers for fat point ideals in the plane: a geometric 
approach.
             {\em Trans. Amer. Math. Soc.}, {\bf 361} (2009), 1103--1127.

\bibitem{h0}B. Harbourne,
 	Global aspects of the geometry of surfaces.
 	{\em Ann. Univ. Paed. Cracov. Stud. Math.}, {\bf 9} (2010), 5--41.

\bibitem{h1}B. Harbourne,
          Problems and progress: survey on fat points in $\mathbb{P}^2$.
          {\em Queens Papers in Pure and Applied Mathematics} {\bf 123}
          (2002), 85--132.

\bibitem{h2} B. Harbourne,
          Complete linear systems on rational surfaces,
          {\em Trans. Amer. Math. Soc.} {\bf 289}, (1985) 213--226.

\bibitem{HFS} B. Harbourne, S. Holay, S. Fitchett,
 		 Resolutions of ideals of quasiuniform fat point 
subschemes of $\mathbb{P}^2$.
 		{\em  Trans. Amer. Math. Soc.}, {\bf  355}  (2003),  no. 
2, 593--608

\bibitem{HMNW} T.~Harima, J.~Migliore, U.~Nagel, J.~Watanabe,
             The weak and strong Lefschetz properties for Artinian $\KK$-algebras.
             {\em J. Algebra}, {\bf 262} (2003), 99--126.

\bibitem{HL} M.~Hochster, D.~Laksov, 
              The linear syzygies of generic forms.
             {\em Comm. Algebra} {\bf 15} (1987) 227--239.

\bibitem{I} A.~Iarrobino,
             Inverse system of a symbolic power, II: the Waring problem for forms. 
             {\em J. Algebra}, {\bf 174} (1995), 1091-1110. 

\bibitem{I2} A.~Iarrobino,
             Inverse system of a symbolic power III: thin algebras and fat points. 
             {\em Compositio Math.}, {\bf 108}, (1997), 319-356.
 
\bibitem{MMN} J.~Migliore, R.~Mir\'o-Roig,  U.~Nagel,
             Monomial ideals, almost complete intersections and the weak 
Lefschetz property.
             {\em Trans. Amer. Math. Soc.},to appear.

\bibitem{MMRN2} J.~Migliore, R.~Mir\'o-Roig,  U.~Nagel,
             On the weak Lefschetz property for powers of linear forms
             preprint, July 2010.

\bibitem{MMgenforms} J.~Migliore, R.~Mir\'o-Roig,
             Ideals of general forms and the ubiquity of the weak Lefschetz property.
             {\em J. Pure Appl. Algebra}, {\bf 182} (2003), 79--107.
\bibitem{m1}
         R. Miranda,
         Linear systems of plane curves.
         {\em  Notices Amer. Math. Soc.}, {\bf 46} (1999), 192--201.

\bibitem{N}
         M. Nagata.
         On rational surfaces, II,
         {\em Mem. Coll. Sci. Univ. Kyoto, Ser. A Math.}, {\bf 33} (1960), 
271--293.

\bibitem{SS} H.~Schenck, A.~Seceleanu,
             The weak Lefschetz property and powers of linear forms in 
$\KK[x,y,z]$.
             {\em Proc. Amer. Math. Soc.}, {\bf 138} (2010) 2335-2339.

\bibitem{Stan} R.~Stanley,
             Weyl groups, the hard Lefschetz theorem, and the Sperner property.
             {\em  SIAM J. Algebraic Discrete Methods}, {\bf  1}  (1980), 168--184.

\bibitem{sturm} B.~Sturmfels, Z.~Xu,
               Sagbi bases of Cox-Nagata rings,
             {\em J. Eur. Math. Soc.}, {\bf  12}  (2010), 429–459.
\end{thebibliography}

\end{document}